\documentclass[11pt]{article}

\usepackage{amssymb,amsthm,amsmath}


\usepackage[usenames,dvipsnames]{xcolor}
\usepackage[colorlinks,citecolor=blue,linkcolor=BrickRed]{hyperref}\usepackage[colorlinks,citecolor=blue,linkcolor=BrickRed]{hyperref}
\usepackage{thmtools,thm-restate}
\usepackage{comment}
\usepackage{cleveref}
\usepackage{fullpage}
\usepackage{pdfpages}
\usepackage{graphicx}
\graphicspath{ {./images/} }

\newtheorem{theorem}{Theorem}[section]
\newtheorem{corollary}[theorem]{Corollary}

\newtheorem{lemma}[theorem]{Lemma}

\newtheorem{definition}[theorem]{Definition}

\newtheorem{proposition}[theorem]{Proposition}
\newtheorem{remark}[theorem]{Remark}
\newtheorem{claim}[theorem]{Claim}

\newtheorem{fact}[theorem]{Fact}
\newtheorem*{claim*}{Claim}

\newcommand{\E}{\mathbb{E}}

\newcommand{\R}{\mathbb{R}}

\newcommand{\Tr}{\mathsf{Tr}}

\newcommand{\diag}{\mathsf{diag}}

\newcommand{\ignore}[1]{{}}

\title{An Exposition of the $\widetilde{O}(\log^{1/4} n)$ Bound for the Koml\'os Problem\footnote{An extended abstract of this work appeared in STOC 2026 \cite{BJ26}. The present article treats only the Koml\'os problem, with the proof simplified and recast in the language of
stochastic calculus. The bounds for the Beck--Fiala problem obtained in \cite{BJ26}
are not reproduced here: their proofs require additional technicalities,
and stronger bounds have since been obtained by Altschuler and Tikhomirov \cite{AT26}
by a different method.}}

\author{Nikhil Bansal\thanks{University of Michigan, Ann Arbor, MI, USA. \texttt{bansaln@umich.edu}.} 
 \and 
Haotian Jiang\thanks{University of Chicago, Chicago, IL, USA. \texttt{jhtdavid@uchicago.edu}.}}

\date{}
\begin{document}

\allowdisplaybreaks
\begin{titlepage}
\maketitle

\begin{abstract}
A conjecture of Koml\'os states that the combinatorial discrepancy of any matrix $A\in\mathbb R^{m\times n}$ whose columns have Euclidean norm at most one is bounded by a universal constant. We prove that the combinatorial discrepancy of every such matrix is at most $O((\log n)^{1/4}(\log\log n)^{7/4})$. This is the first asymptotic improvement over the $O(\sqrt{\log n})$ bound established by Banaszczyk [Banaszczyk, Random Struct.\ Algorithms, 1998], and it refutes a conjecture of Hajela [Hajela, European J.\ Combin., 1988] that a lower bound of order $\Omega(\sqrt{\log n})$ should hold.
\end{abstract}

 \thispagestyle{empty}
\end{titlepage}

\newpage
\setcounter{page}{1}

\section{Introduction}
\label{sec:introduction}

The combinatorial discrepancy of a matrix $A \in \R^{m \times n}$ is the smallest $\ell_\infty$ norm over all possible signed sums of its columns:
\[
\mathsf{disc}(A):=\min_{x\in\{-1,1\}^n}\|Ax\|_\infty .
\]
A central open problem in discrepancy theory is a conjecture of Koml\'os from the 1980s which asserts that $\mathsf{disc}(A)$ is bounded by a universal
constant whenever every column of $A$ has Euclidean norm at most one.
The Koml\'os conjecture contains, after a normalization, the Beck--Fiala conjecture: if
$A\in\{0,1\}^{m\times n}$ is the incidence matrix of a set system in which
every element lies in at most $k$ sets, then $k^{-1/2}A$ is a Koml\'os
instance, so the Koml\'os conjecture would imply
$\mathsf{disc}(A)=O(\sqrt k\,)$.  Beck and Fiala \cite{BF81} proved that
$\mathsf{disc}(A)\le 2k-1$ for every such set system and conjectured the
$O(\sqrt k\,)$ bound.  We refer to \cite{Cha00,Mat09,CST14,Ban22} for background on
discrepancy theory and its connections with many other areas of mathematics and computer science. 

The partial-coloring method, introduced by Beck \cite{Bec81} and refined
by Spencer \cite{Spe85} and Gluskin \cite{Glu89}, colors a constant fraction
of the coordinates at bounded discrepancy; iterating it for $O(\log n)$
rounds yields $\mathsf{disc}(A)=O(\log n)$ for the Koml\'os problem.
The best prior bound for this problem was $O(\sqrt{\log n})$ due to Banaszczyk \cite{Ban98}.

In the opposite direction, no lower bound growing with $n$ is known: the best known lower bound, due to Kunisky \cite{Kun23}, provides instances with
discrepancy approaching $1+\sqrt 2$.  Hajela \cite{Haj88} nevertheless
conjectured that a lower bound of order $\Omega(\sqrt{\log n})$ should hold (matching the upper bound later proved by Banaszczyk), and
established a weaker form of it in which the minimum is taken over an
arbitrary, but subexponential, family of sign vectors in $\{-1,1\}^n$; see also
\cite{CS21}.

We give an asymptotic improvement on Banaszczyk's bound, thereby refuting Hajela's conjecture.

\begin{restatable}{theorem}{maintheorem}\label{thm:main}
There is a universal constant $C>0$  such that for any matrix $A\in\mathbb R^{m\times n}$ whose columns have Euclidean norm at most one, there exists $x\in\{-1,1\}^n$ such that
\[
    \|Ax\|_\infty
    \le C(\log n)^{1/4}(\log\log n)^{7/4}.
\]
\end{restatable}

To keep the exposition clean, we have made no attempt to optimize the $(\log \log n)^{7/4}$ factor.

\subsection{Outline of the Proof}
The coloring $x \in \{-1,1\}^n$ in \Cref{thm:main} is produced via a stochastic process.  Let $(W_t)_{t \geq 0}$ be a
standard Brownian motion in $\R^n$, and consider the It\^o process
$(x_t)_{0 \leq t \leq T} \subset [-1,1]^n$ given by
\[
    x_0 = 0, \qquad d x_t = \Sigma_t \, d W_t ,
\]
where the progressively measurable diffusion matrix $\Sigma_t \in \R^{n \times n}$ is chosen adaptively to satisfy various desirable properties and to ensure 
that the process is \emph{sticky}: each coordinate  $x_t(j)$ ceases to move once it reaches $\pm 1$.
We defer the details to \Cref{subsec:ito_Komlos}.

Roughly speaking, prior approaches \cite{BG17,BDG19,BLV22} choose $\Sigma_t$ so that the rate of the quadratic variation $U_t := d [x,x]_t / d t = \Sigma_t \Sigma_t^\top$ is $O(1)$-{\em spectrally independent}, meaning that $U_t \preceq O(1) \cdot \diag(U_t)$. 
This guarantees that the discrepancy of each row $a_i$ has a subgaussian tail, and  a union bound over
the rows recovers Banaszczyk's $O(\sqrt{\log n})$ bound. 
We reproduce this proof in \Cref{subsec:Banaszczyk_alg}.
It is important to note that this extra $\sqrt{\log n}$ factor due to the union bound is unavoidable so long as the rows are controlled one at a
time.

To bypass the $\sqrt{\log n}$ bottleneck, our new idea is to impose spectral independence also on the increments of certain row
discrepancies themselves, a property we call \emph{affine spectral independence}.
Under this new property, we prove a {\em decoupling bound} (\Cref{thm:decoupling-bound}), which, together with a counting argument, shows that at any point in the process, there cannot be too many {\em dangerous} rows with both high discrepancy and many coordinates that have not reached $\pm 1$ (\Cref{lem:key-lemma-Komlos}).  
This allows us to control their discrepancies by walking orthogonally to all of them. 

\smallskip
\noindent \textbf{Organization.}
\Cref{sec:prelim} presents basic notation and definitions, and   develops the technical tools used in the proof: the
decoupling bound for vector-valued It\^o processes appears in \Cref{subsec:decoupling}, its application to discrepancy is explained in \Cref{subsec:regularized_disc}, the subgaussian
discrepancy estimate corresponding to Banaszczyk's bound is given in \Cref{subsec:Banaszczyk_alg}, and the existence of $\Sigma_t$ under various constraints is shown in \Cref{subsec:feasibility_SDP}.  The proof of \Cref{thm:main} is then given in \Cref{sec:komlos_proof}.

\section{Preliminaries}
\label{sec:prelim}

\smallskip
\noindent \textbf{Notation.}
For a vector $u \in \R^n$, we use $u(j)$ to denote its $j$th entry, and $u^{\odot 2}$ to denote the vector with entries $u(j)^2$ for $j\in [n]$. 
We write the coordinate-wise product of two vectors $u,v\in \R^n$ as $u \odot v$. 
For a matrix $M$, we use $M(i,j)$ to denote its $(i,j)$-th entry, $M(i,\cdot)$ to denote its $i$th row, $\Tr(M)$ to denote the trace of $M$, and $\|M\|_{\mathsf{HS}}$ to denote its Hilbert-Schmidt norm.
We use PSD to mean positive semi-definite. 
Throughout, $\log$ is to the base $e$, unless otherwise stated. 
We use $a_1,\ldots,a_m$ to denote the rows of the input matrix $A$, and $a_i(j) = A(i,j)$ to denote its $j$th entry. Throughout, we 
use $i$ to index the rows and $j$ to index columns.

\subsection{It\^o Calculus}

In this subsection, we recall the basics of It\^o calculus. 

\begin{definition}[It\^o processes] 
An adapted continuous $\mathbb R^m$-valued process $Z=(Z_t)_{0\le t\le T}$ 
is called an \emph{It\^o process} if there exist progressively measurable processes $\Sigma_t\in\mathbb R^{m\times r}$ (for some $r\ge 1$) and $b_t\in\mathbb R^m$  satisfying
\[
\int_0^T \|\Sigma_t\|_{\mathsf{HS}}^2\,dt < \infty
\qquad\text{and}\qquad
\int_0^T \|b_t\|_1\,dt < \infty
\quad\text{almost surely,}
\]
such that $d Z_t = \Sigma_t\,dW_t + b_t\,dt$, where $W_t$ is a standard Brownian motion in $\R^{r}$. 
\end{definition}

The quadratic variation of $Z_t$, denoted by $[Z,Z]_t$, is given by $d [Z,Z]_t = \Sigma_t \Sigma_t^\top d t$. 
The increments for functions of an It\^o process are given by the It\^o formula.

\begin{theorem}[It\^o formula] \label{thm:Ito-formula}
Let $Z=(Z_t)_{0\le t\le T}$ be an $\mathbb R^m$-valued It\^o process with increment $dZ_t=\Sigma_t\,dW_t+b_t\,dt$, 
and let $f: \R^m \rightarrow \R$ be a twice continuously differentiable function. Then the real-valued process $(f(Z_t))_{0\le t\le T}$ is also an It\^o process whose increment is given by 
\[
df(Z_t)
= \nabla f(Z_t)^\top\, dZ_t + \frac12 \Tr\bigl(\nabla^2 f(Z_t)\, d[Z,Z]_t\bigr), 
\]
where $\nabla f$ denotes the gradient of $f$ and $\nabla^2f$ its Hessian matrix.
\end{theorem}

A real-valued It\^o process is a {\em supermartingale} if $b_t \leq 0$ for all $0 \leq t \leq T$ almost surely. A random time $\tau$ is called a \emph{stopping time} if, for each time $t\ge 0$, the event $\{\tau\leq t\}$ is fully determined by the information up to time $t$. We need the following optional stopping theorem for supermartingales. 

\begin{theorem}[Optional stopping theorem]
Let $(Z_t)_{0\le t\le T}$ be a supermartingale, and let $\tau$ be a stopping time such that $0 \le \tau \le T$ almost surely. Then $\E[Z_\tau] \leq \E[Z_0]$. 
\end{theorem}

\subsection{$(\alpha,\theta)$-Independent It\^o Process and a Decoupling Bound}
\label{subsec:decoupling}
We first recall a basic large deviation bound for
a real-valued It\^o process with negative drift.
A discrete analogue of this bound appears as Lemma~2.2 of \cite{Ban24}.

\begin{proposition}[Freedman-type inequality]
\label{prop:freedman_Ito}
Let $(Z_t)_{0\le t\le T}$ be a real-valued It\^o process with increment $dZ_t = \sigma_t^{\top}\,dW_t + b_t\,dt$, where $\sigma_t$ may be vector-valued. Let $\delta >0$ and 
suppose that $b_t \le -\delta \|\sigma_t\|_2^2$  for all $t\in[0,T]$ almost surely. Then, for every $\xi\ge 0$,
\[
\Pr \Big[\sup_{0\le t\le T}(Z_t-Z_0)\ge \xi\Big]\le \exp(-2\delta \xi).
\]
\end{proposition}
\begin{proof}
We may assume that $Z_0$ is deterministic, by conditioning on its value. Set $\lambda:=2\delta$ and define $X_t := \exp(\lambda Z_t)$ for $0\le t\le T$. 
By It\^o's formula, 
\[
dX_t=
\lambda  X_t\,dZ_t+ \frac{\lambda^2}{2} X_t\,d [Z,Z]_t = \lambda X_t \sigma_t^{\top}\,dW_t
+\lambda X_t\left( b_t+ \delta \|\sigma_t\|_2^2\right)dt.
\]
As $b_t\le -\delta \|\sigma_t\|_2^2$, it follows that $X_t$ is a supermartingale.  
Let $
\tau := \inf\{t\in[0,T]: Z_t-Z_0\ge \xi\}$. 

By the optional stopping theorem applied to $\tau\wedge T$, we have $\mathbb E[X_{\tau\wedge T}] \le \mathbb E[X_0]=\exp(\lambda Z_0)$.
As $X_{\tau\wedge T} = \exp(\lambda(Z_0+\xi))$ on the event $\{\tau \le T\}$, 
and $X_{\tau \wedge T} \geq 0$ always,  
this gives
\[
\exp(\lambda(Z_0+\xi))\,\Pr (\tau \le T) \leq \mathbb E[X_{\tau\wedge T}] \le \exp(\lambda Z_0).
\] 
Dividing by $\exp(\lambda(Z_0+\xi))$ finishes the proof, as $\{\tau\le T\} = \{\sup_{0\le t\le T}(Z_t-Z_0)\ge \xi\}$ and $\lambda = 2\delta$.
\end{proof}

We will be interested in vector-valued processes $Z =(Z_t)_{0\le t\le T} \subset \R^m$, where
each coordinate $Z_t(i)$ has a negative drift, and the martingale parts evolve almost independently in the following sense.

\begin{definition}[$(\alpha,\theta)$-independent It\^o process] \label{def:alpha-theta-independence}
Let $Z =(Z_t)_{0\le t\le T} \subset \R^m$ be an It\^o process
with increment $dZ_t = \Sigma_t\,dW_t + b_t\,dt$.  
Write $C_t := \Sigma_t\Sigma_t^\top \in \mathbb R^{m\times m}$.
Let $\alpha\ge 1$ and $\theta>0$. We say that $Z$ is an \emph{$(\alpha,\theta)$-independent
It\^o process} if the following holds for all $t\in[0,T]$ almost surely:
\begin{enumerate}
\item[(i)] (\emph{Almost pairwise independence}) $C_t \preceq \alpha \, \diag(C_t)$.
\item[(ii)] (\emph{Coordinate-wise negative drift}) For every $i\in[m]$, we have $b_t(i)\le -\theta\, C_t(i,i)$.
\end{enumerate}
\end{definition}
We call the first condition almost pairwise independence since it is equivalent to saying that for all  test vectors $u\in \R^m$,
the martingale part $dM_t = \Sigma_t dW_t$ of $dZ_t$ satisfies 
\[ \E \big[\big(\sum_i u(i) dM_t(i)\big)^2\big] =  \E[ (u^{\top} dM_t)^2] = u^{\top} C_t u dt \leq \alpha \, u^{\top}\diag(C_t)udt = \alpha \sum_i u(i)^2 \E[dM_t(i)^2].\] 
Also note that any process $Z_t$ that satisfies $d Z_t = 0$ for all $0 \leq t \leq T$ is clearly an  $(\alpha,\theta)$-independent It\^o process, for any $\alpha \geq 1$ and $\theta > 0$.

We prove the following decoupling bound for such $(\alpha,\theta)$-independent It\^o processes.

\begin{theorem}[Decoupling bound]  \label{thm:decoupling-bound}
Let $Z =(Z_t)_{0\le t\le T}$ be an $\R^m$-valued $(\alpha,\theta)$-independent
It\^o process with increment $d Z_t = \Sigma_t d W_t + b_t d t$. 
Let $B>0$, 
and call a coordinate $i\in[m]$ \emph{bad} if
\[
\sup_{0\le t\le T} \big(Z_t(i) - Z_0(i) \big) \ge B.
\]
Let $N_{\mathsf{bad}}$ be the number of bad coordinates.
Then for every $0 < \lambda \leq \theta$
and every $0 < \gamma < 1$, 
\[
\Pr \Big[
N_{\mathsf{bad}}
\ge
m e^{-\lambda B}
+
\frac{\alpha\lambda}{\theta}\log(1/\gamma)\Big]
\le \gamma.
\]
\end{theorem}
\begin{remark}
Observe that the parameter $\alpha$ only affects the additive term in the bound for $N_{\mathsf{bad}}$, and not the first term $m e^{-\lambda B}$, which is roughly the expected number of bad coordinates, even if the martingale parts are fully independent (i.e., $\alpha=1$). We will use this crucially in our results later.
\end{remark}
\begin{proof} 
For each coordinate $i\in[m]$, define the stopping time $
\tau_i := \inf\{t: Z_t(i) - Z_0(i)\ge B\}\wedge T.$
Fix $0 < \lambda \leq \theta$, and define
\[
X_t(i):=\exp\Bigl(\lambda \big(Z_{t\wedge \tau_i}(i) - Z_0(i) \big)\Bigr) \quad \text{and} \quad 
\Phi_t:=\sum_{i=1}^m X_t(i).
\] 
Clearly $X_t(i)\leq e^{\lambda B}$, and a bad coordinate $i$ contributes exactly $X_T(i)=e^{\lambda B}$ to $\Phi_T$. Therefore, 
\begin{align} \label{eq:N_bad_by_Phi_T}
N_{\mathsf{bad}} \le e^{-\lambda B} \Phi_T.
\end{align}
Our goal will be to bound $\Phi_T$ by applying \Cref{prop:freedman_Ito} suitably.

Let $M_t(i):=\int_0^t \Sigma_s(i,\cdot)\,dW_s$ be the martingale part of $Z_t(i)$, so that $dZ_t(i)=dM_t(i)+b_t(i)\,dt$. 
Applying It\^o's formula to 
$Z_{t\wedge\tau_i}(i)$ gives,
\[
dX_t(i)
=
\mathbf 1_{\{t<\tau_i\}} X_t(i)
\left( \lambda\, dM_t(i) + \Bigl(\lambda b_t(i)+\frac{\lambda^2}{2}C_t(i,i)\Bigr)dt
\right), 
\]
where $C_t = \Sigma_t \Sigma_t^\top$ as before.  
Summing over $i$ gives the decomposition
$
d \Phi_t = dM_t^{(\Phi)} + b_t^{(\Phi)}\,dt,
$
where
\[
dM_t^{(\Phi)}
=
\lambda\sum_{i=1}^m w_i\, dM_t(i) \quad \text{and} \quad  b_t^{(\Phi)} = \sum_{i=1}^m
w_i
\Bigl(\lambda b_t(i)+\frac{\lambda^2}{2}C_t(i,i)\Bigr),
\]
where we write $w_i  := w_i(t) = \mathbf 1_{\{t<\tau_i\}} X_t(i)$ for all $i \in [m]$ for ease of notation.

As $Z$ is $(\alpha, \theta)$-independent, it satisfies $b_t(i)\le -\theta C_t(i,i)$. Further, as $0 < \lambda \leq \theta$, it follows that 
\[
\lambda b_t(i)+\frac{\lambda^2}{2} C_t(i,i)
\le
\left(-\lambda\theta+\frac{\lambda^2}{2}\right) C_t(i,i) \leq  -\frac{\lambda\theta}{2} C_t(i,i).
\]
Therefore, $\Phi_t$ has a negative drift 
\begin{align} \label{eq:Phi-negative_drift}
b_t^{(\Phi)}
\le
-\frac{\lambda\theta}{2}
\sum_{i=1}^m \mathbf 1_{\{t<\tau_i\}} X_t(i)\, C_t(i,i) = -\frac{\lambda\theta}{2} \sum_i w_i C_t(i,i).
\end{align}
We now bound the quadratic variation of $M_t^{(\Phi)}$.  As $d[ M(i),M(j)]_t = C_t(i,j)\,dt$, we have
\begin{align} \label{eq:Phi-quadratic_var}
d[M^{(\Phi)},M^{(\Phi)}]_t
& 
= \lambda^2 \sum_{i,j=1}^m w_iw_j C_t(i,j) d t  \leq \lambda^2 \alpha \sum_{i=1}^m w_i^2 C_t(i,i) d t \leq \alpha\lambda^2 e^{\lambda B}
\sum_{i=1}^m w_i C_t(i,i) dt ,
\end{align} 
where the first inequality follows as $C_t \preceq \alpha \cdot \diag(C_t)$, and the second as $w_i \leq X_t(i)\leq  e^{\lambda B}$ for all $t$.

Combining  
\eqref{eq:Phi-negative_drift} and \eqref{eq:Phi-quadratic_var}, we get that  
$\Phi_t$ satisfies 
$b_t^{(\Phi)}\,dt
\le
-\delta\, d[M^{(\Phi)},M^{(\Phi)}]_t$ for 
$\delta:=\theta/(2\alpha\lambda e^{\lambda B})$.
Thus \Cref{prop:freedman_Ito} gives that for every $\xi\ge 0$,
\[
\Pr \Big[\sup_{0\le t\le T}(\Phi_t- \Phi_0)\ge \xi\Big]
\le e^{-2\delta \xi}.
\]
As $\Phi_0 = m$,
setting $\xi := (1/2\delta)\log (1/\gamma) =
(\alpha\lambda e^{\lambda B}/\theta)\log (1/\gamma)$ gives 
\[
\Pr \Big[\Phi_T \ge m + \frac{\alpha\lambda e^{\lambda B}}{\theta}\log (1/\gamma) \Big]
\leq \gamma.
\]
Using the bound $N_{\mathsf{bad}}\le e^{-\lambda B} \Phi_T$ in \eqref{eq:N_bad_by_Phi_T}, we obtain the claimed result.
\end{proof}

\subsection{Regularized Discrepancy}
\label{subsec:regularized_disc}

Let $(x_t)_{0 \leq t \leq T} \subset [-1,1]^n$ be an It\^o process (for the coloring) with increment $dx_t = U_t^{1/2}\,dW_t$,  where $W_t$ is a standard Brownian motion and $U_t$ are progressively measurable PSD matrices.

Instead of tracking how the discrepancy vector $Ax_t$ evolves with $x_t$, it will be convenient to work with a ``regularized'' discrepancy. 
In particular, under suitable conditions on $U_t$, the regularized discrepancy vector will be an $(\alpha,\theta)$-independent process for some choices of $\alpha$ and $\theta$.

Define the {\em regularized discrepancy} vector $Y_t \in \R^m$ as follows. For each $i \in [m]$, define 
\begin{align} \label{eq:reg-disc-prelim}
Y_t(i) := \langle a_i,x_t\rangle + \beta \sum_{j=1}^n a_i(j)^2\bigl(1-x_t(j)^2\bigr) ,
\end{align}
where $\beta \ge 0$ is a parameter. 
It is convenient to introduce a matrix $E_t \in \R^{m \times n}$, a slight modification of $A$, whose $i$th row is given by 
\begin{align} \label{eq:E_t-reg-disc}
E_{t}(i,\cdot):=a_i-2\beta\,(a_i^{\odot 2}\odot x_t)\in\mathbb R^n.
\end{align}
For $dx_t = U_t^{1/2}dW_t$, we have the following.
\begin{lemma}[Increment of $Y_t$] \label{lem:reg-disc-increment}
The process $Y_t$ is an It\^o process with increment 
\[
dY_t(i) = \langle E_t(i,\cdot), \, U_t^{1/2} dW_t \rangle -\beta\langle a_i,\diag(U_t) a_i\rangle\,dt.
\]
Consequently, its quadratic variation has increment $d [Y,Y]_t = E_tU_t E_t^\top dt$. 
\end{lemma}

\begin{proof}
Using the It\^o formula in \Cref{thm:Ito-formula}, we can compute
\begin{align*}
d Y_t(i) 
& = \langle a_i, \, dx_t  \rangle - 2 \beta \sum_{j=1}^n a_i(j)^2 x_t(j) d x_t(j) - \beta \sum_{j=1}^n a_i(j)^2  (d x_t(j))^2 \\
& =\sum_{j=1}^n \left(( a_i(j) - 2 \beta  a_i(j)^2 x_t(j)) (U_t^{1/2} dW_t)(j)\right) - \beta \sum_{j=1}^n a_i(j)^2  U_t(j,j) d t \\
& = \langle a_i - 2 \beta (a_i^{\odot 2} \odot x_t), \, U_t^{1/2} dW_t\rangle - \beta  \langle a_i,\diag(U_t) a_i\rangle\,dt , 
\end{align*}
which gives the first statement. 
The quadratic variation formula follows immediately. 
\end{proof}

\begin{proposition}[Weak independence of $Y_t$] \label{prop:weak_ind_reg-disc} 
Suppose that $U_t$ satisfies the following two conditions for all $0\leq t \leq T$, almost surely.  
\begin{enumerate}
\item[(i)] (Affine spectral independence) $E_tU_tE_t^\top \preceq \alpha \, \diag(E_tU_tE_t^\top)$ , 
\item[(ii)] (Spectral independence) $U_t \preceq \eta \, \diag(U_t)$.
\end{enumerate}
Moreover, suppose that $\beta \|a_i\|_\infty \leq 1/2$ for all $i \in [m]$.
Then
 \(Y_t\) is an $(\alpha,\frac{\beta}{4\eta})$-independent It\^o process.
\end{proposition}
\begin{proof}
Recall the definition of $(\alpha,\theta)$-independent It\^o process in \Cref{def:alpha-theta-independence}. To verify the almost pairwise independence condition, note that by \Cref{lem:reg-disc-increment}, the matrix $C_t$ in \Cref{def:alpha-theta-independence} is given by $C_t = d [Y,Y]_t/dt = E_t U_t E_t^\top$, which clearly satisfies $C_t \preceq \alpha \, \diag(C_t)$ by the assumption (i) above. 

We now verify the second condition (coordinate-wise negative drift).
By \Cref{lem:reg-disc-increment}, the $C_t(i,i)$ in \Cref{def:alpha-theta-independence} is given by
\[C_t(i,i) = d [Y(i), Y(i)]_t / dt = E_t(i,\cdot)^\top U_t E_t(i,\cdot) \leq \eta \,E_t(i,\cdot)^\top \diag(U_t) E_t(i,\cdot),\]
where the inequality follows by assumption (ii), and the drift $b_t(i)$ 
is given by $b_t(i) = - \beta a_i^\top \diag(U_t) a_i$.

Since $\beta \|a_i\|_\infty \leq 1/2$ and $\|x_t\|_\infty \leq 1$, each coordinate $j \in [n]$ of the vector $E_t(i,\cdot)$ satisfies
\[
|E_t(i,j)| = |a_i(j)| \cdot |(1 - 2 \beta a_i(j) x_t(j))| \leq 2 |a_i(j)| .
\]
Thus, 
$E_t(i,\cdot)^\top \diag(U_t) E_t(i,\cdot) \leq 4 a_i^\top \diag(U_t) a_i$,
 and we obtain that 
\[
b_t(i) = - \beta a_i^\top \diag(U_t) a_i \leq - \frac{\beta}{4} E_t(i,\cdot)^\top \diag(U_t) E_t(i,\cdot) \leq -\frac{\beta}{4 \eta} E_t(i,\cdot)^\top U_t E_t(i,\cdot) = - \frac{\beta}{4 \eta} C_t(i,i). 
\qedhere
 \]
\end{proof}

As an immediate corollary, \Cref{thm:decoupling-bound} implies the following decoupling bound for $Y_t$.

\begin{corollary}[Decoupling bound for $Y_t$] \label{cor:decoupling_reg-disc}
Under the assumptions of \Cref{prop:weak_ind_reg-disc}, we have the following.
Let $B > 0$, and denote by $N_{\mathsf{bad}} := \big|\bigl\{i\in[m]: \sup_{0\le t\le T} \big(Y_t(i) - Y_0(i) \big)\ge B\bigr\} \big|$ the number of bad coordinates. 
Then for every $0 < \lambda \leq \beta/(4 \eta)$
and $\gamma>0$, one has
\[
\Pr \Big[
N_{\mathsf{bad}}
\ge
m e^{-\lambda B}
+
\frac{4 \eta \alpha\lambda}{\beta}\log(1/\gamma)\Big]
\le \gamma.
\]
\end{corollary}

\subsection{A Subgaussian Discrepancy Bound Attaining Banaszczyk's Result}
\label{subsec:Banaszczyk_alg}
As is well known, the spectral independence property (condition (ii) in Proposition \ref{prop:weak_ind_reg-disc}) already gives a subgaussian tail bound for the discrepancy $\langle a, x_t \rangle$ for any row. This bound corresponds to Banaszczyk's $O(\sqrt{\log n})$ bound for the Koml\'os problem. 

\begin{proposition}[Subgaussian discrepancy via spectral independence] \label{prop:spec-indep-subg-disc}
Let $(x_t)_{0\leq t\leq T}$ be an It\^o process with increment $dx_t = U_t^{1/2}\,dW_t$. If   $U_t \preceq \eta \, \diag(U_t)$ for every $0\leq t \leq T$, 
then for every $\gamma>0$ and every vector $ a\in\mathbb R^n$,
\[
\Pr \Big[
\sup_{0\leq t \leq T}\langle a, x_t-x_0\rangle
\ge
C_0 \,\|a\|_2\sqrt{\eta \log(1/\gamma)}
\Big]\leq \gamma,
\]
where $C_0 > 0$ is a universal constant. 
\end{proposition}

\begin{proof}
By rescaling, we can assume that $\|a\|_2=1$. Let us decompose $a=a_{\mathsf{big}}+a_{\mathsf{small}}$ where 
\[
a_{\mathsf{big}}(j):=a(j)\,\mathbf 1_{\{|a(j)|\ge \lambda^{-1}\}} \quad \text{and} \quad 
a_{\mathsf{small}}(j):=a(j)\,\mathbf 1_{\{|a(j)|< \lambda^{-1}\}} ,
\]
where the threshold $\lambda$ will be fixed later.

We first bound the contribution due to $a_{\mathsf{big}}$.
As  $\|a_{\mathsf{big}}\|_2 \leq \|a\|_2=1$, the support  of 
\(a_{\mathsf{big}}\) has size at most $\lambda^2$. Since $|x_t(j)-x_0(j)|\le 2$ for all $j$, applying Cauchy--Schwarz gives that
\begin{align}
\label{eq:big-contrib}
\sup_{0 \leq t \leq T} \langle a_{\mathsf{big}}, x_t-x_0\rangle
\le 2 \|a_{\mathsf{big}}\|_1 \leq 
2 \|a_{\mathsf{big}}\|_2  \cdot   |\mathsf{supp}(a_{\mathsf{big}})|^{1/2}\leq 2\lambda.
\end{align}
We now consider $a_{\mathsf{small}}$. 
Consider the ($1$-d) regularized discrepancy
\[
Y_t := \langle a_{\mathsf{small}},x_t\rangle + \beta \sum_{j=1}^n a_{\mathsf{small}}(j)^2\bigl(1-x_t(j)^2\bigr),\]
where we set $\beta = \lambda/2$ 
(so that $\beta \|a_{\mathsf{small}}\|_\infty \leq 1/2$ as needed in \Cref{prop:weak_ind_reg-disc}). 

Since $U_t \preceq \eta\,\diag(U_t)$, \Cref{prop:weak_ind_reg-disc} implies that $Y_t$ is a $(1,\beta/4\eta)$-independent It\^o process (assumption (i) in \Cref{prop:weak_ind_reg-disc} is satisfied trivially with $\alpha=1$ as $Y_t$ is $1$-dimensional).

Thus, applying \Cref{prop:freedman_Ito} with $\delta = \beta/(4\eta)$ and $\xi =(2\eta/\beta)\log(1/\gamma)$, we have that 
\[
\Pr\Big[\sup_{0\le t\le T}(Y_t- Y_0)\ge \frac{2\eta}{\beta} \log(1/\gamma)\Big] \leq \gamma. 
\]
As $\langle a_{\mathsf{small}}, x_t \rangle \leq  Y_t \leq \langle a_{\mathsf{small}}, x_t \rangle + \beta$ for any time $t$, this implies that with probability at least $1-\gamma$, 
\begin{align*}\label{eq:bana_small}
\sup_{0 \leq t \leq T} \langle a_{\mathsf{small}}, x_t - x_0\rangle \leq \beta + \frac{2\eta}{\beta} \log (1/\gamma) 
\end{align*}
 Together with \eqref{eq:big-contrib},  setting $\lambda = 2 \beta = \sqrt{\eta \log(1/\gamma)}$ gives the result.
\end{proof}

\subsection{Existence of Diffusion Matrices}
\label{subsec:feasibility_SDP}

The diffusion matrix $\Sigma_t$ in our It\^o process will be chosen to satisfy various conditions needed for the analysis. 
We show the existence of such $\Sigma_t$ via the feasibility of a semi-definite program (SDP) with matrix variable $U$, from which $\Sigma_t$ is obtained by normalizing $U^{1/2}$ (cf.\ \Cref{def:Sigma_t-Komlos}). 

\begin{restatable}[SDP Feasibility]
{theorem}{SubIsoVecMultiClass}\label{thm:sub-isotropic-SDP_Multi-Class}
Let $H \subset \R^h$ be a subspace with $\mathsf{dim}(H) \leq \delta h$, and let $E_s \in \mathbb{R}^{m_s \times h}$, for $s \in [q]$, be matrices with $m_s \leq r_s h$ rows, for some real numbers $r_s > 0$. 
Then for any $\kappa, \eta, \alpha_s > 0$, where $s \in [q]$, such that $\eta^{-1} + \kappa + \sum_{s=1}^q r_s \alpha_s^{-1} \leq 1-\delta$, there is an $h \times h$ PSD matrix $U$ satisfying the following:

(i) $\langle vv^\top, U \rangle = 0$ for all $v \in H$, 

(ii) $U(i,i) \leq 1$ for all $i \in [h]$, 

(iii) $\Tr(U) \geq \kappa h$, 

(iv) $U \preceq \eta \, \diag(U)$, and 

(v) $E_sU E_s^\top \preceq \alpha_s \, \diag(E_s U E_s^\top)$ for all $s \in [q]$.
\end{restatable}

These conditions on $U$ correspond to the conditions in the definition of the diffusion matrix $\Sigma_t$ in \Cref{def:Sigma_t-Komlos}, where we also explain what each constraint implies about the process. 

\begin{proof}
Consider the following SDP
\begin{equation}
\label{eq:sdp_primal} \tag{Primal SDP}
\begin{aligned}
    \max \quad & \Tr(U) \\
    \textrm{s.t.} \quad & \langle U ,\, vv^\top\rangle = 0 \ , && \text{for all } v \in H ,\\
    & \langle U ,\, e_i e_i^\top\rangle \leq 1 \ , &&\text{for all } i \in [h] , \\
    & U \preceq \eta\, \diag(U) , \\
    & E_s U E_s^\top \preceq \alpha_s \diag(E_s U E_s^\top ) \ , \qquad && \text{for all } s \in [q] , \\
    & U \succeq 0 .
\end{aligned}
\end{equation}
It suffices to show that the optimal value of \eqref{eq:sdp_primal} is at least $\kappa h$, and we do so by considering the dual SDP. 
Let $\gamma_v \in \R$, $\lambda_i \geq 0$, $F \succeq 0$, and $G_s \succeq 0$ be the Lagrangian multipliers for the constraints of \eqref{eq:sdp_primal}. 
Note that $F \in \R^{h\times h}$ and $G_s \in \R^{m_s \times m_s}$.

Then the dual SDP is given by 
\begin{equation}
\label{eq:sdp_dual} \tag{Dual SDP}
\begin{aligned}
& \min \quad \sum_{i \in [h]} \lambda_i \\
& \textrm{s.t.} \sum_{v \in H} \gamma_v vv^\top + \sum_{i \in [h]} \lambda_i e_i e_i^\top + F - \eta\, \diag(F) + \sum_{s \in [q]} (E_s^\top G_s E_s - \alpha_s E_s^\top \diag(G_s) E_s) \succeq I , \\
& \quad F \succeq 0 , \  G_s \succeq 0 \text{ for all $s \in [q]$} , \ \text{and } \lambda_i \geq 0 \text{ for all } i \in [h] .
\end{aligned}
\end{equation}
Here and in \eqref{eq:neg_subspace_SDP} below, the sum over $v$ ranges over a fixed basis of $H$.

Note that strong duality holds here, as \eqref{eq:sdp_dual} is strictly feasible (e.g., \cite[Section 5.9]{BV04book}). For instance, the dual feasible solution $F = G_s = I$ for all $s \in [q]$, $\gamma_v = 0$ for all $v \in H$, and $\lambda_i = \sum_{s \in [q]} 2 \alpha_s \Tr(E_s^\top E_s) + 2\eta$ for all $i \in [h]$ is in the interior of the dual feasible region. 

Our goal is to prove that for an arbitrary feasible solution to \eqref{eq:sdp_dual}, its objective value is at least $\kappa h$.
By strong duality, this would imply that there must be a solution to \eqref{eq:sdp_primal} with objective value at least $\kappa h$.

To do so, we recall the following fact from \cite{BG17}.
\begin{fact}[\cite{BG17}] \label{fact:subspace_diag}
There exists a subspace $W_F \subset \R^h$ with dimension $\dim(W_F) \geq (1 - \eta^{-1}) h$ such that for all vectors $u \in W_F$, one has  $u^\top F u \leq \eta\, u^\top \diag(F) u$ .
\end{fact}
We slightly generalize this fact in the following claim, whose proof is deferred to later.
\begin{claim} \label{claim:subspace_diag_general}
For any $s \in [q]$, there exists a subspace $W_s \subset \R^h$ with dimension $\dim(W_s) \geq (1 - r_s \alpha_s^{-1}) h$ such that for all vectors $u \in W_s$, one has  $u^\top E_s^\top G_s E_s u \leq \alpha_s u^\top E_s^\top \diag(G_s) E_s u$ . 
\end{claim}
Using \Cref{fact:subspace_diag} and \Cref{claim:subspace_diag_general}, we find subspaces $W_F$ of dimension at least $(1 - \eta^{-1})h$ and $W_s$ for all $s \in [q]$ of dimensions at least $(1 - r_s \alpha_s^{-1})h$ each with the corresponding guarantees.
This means for any vector  $u \in W_{\mathsf{neg}} := H^\perp \cap W_F \cap \big(\bigcap_{s =1}^q W_s \big)$, we have
\begin{align} \label{eq:neg_subspace_SDP}
u^\top \Big( \sum_{v \in H} \gamma_v vv^\top + F - \eta\, \diag(F) + \sum_{s \in [q]} \big(E_s^\top G_s E_s - \alpha_s E_s^\top \diag(G_s) E_s \big)  \Big) u \leq 0 .
\end{align}
Note that
\[
\dim(W_{\mathsf{neg}}) \geq \big (1 - \delta - \eta^{-1} - \sum_{s \in [q]} r_s \alpha_s^{-1} \big) h \geq \kappa h ,
\]
and hence there exists a set of $\lceil\kappa h\rceil$ orthonormal vectors $v_1, \cdots, v_{\lceil\kappa h\rceil} \in W_{\mathsf{neg}}$. Then we have
\begin{align*}
\sum_{i \in [h]} \lambda_i  \;\geq\; \sum_{j \in [\lceil\kappa h\rceil]} \Big\langle \sum_{i \in [h]} \lambda_i e_i e_i^\top ,\, v_j v_j^\top \Big\rangle \;\geq\; \sum_{j \in [\lceil\kappa h\rceil]} \big\langle I ,\, v_j v_j^\top \big\rangle \;=\; \lceil\kappa h\rceil \;\geq\; \kappa h ,
\end{align*}
where the first inequality uses that $\sum_{j \in [\lceil\kappa h\rceil]} v_j v_j^\top \preceq I$, the second inequality follows from the constraint of \eqref{eq:sdp_dual} and the inequality \eqref{eq:neg_subspace_SDP}. 
This shows that \eqref{eq:sdp_dual} has value at least $\kappa h$, which proves the theorem. 
\end{proof}

We are now left to prove  \Cref{claim:subspace_diag_general}. 

\begin{proof}[Proof of \Cref{claim:subspace_diag_general}]
We assume without loss of generality that all diagonal entries of $G_s$ are strictly positive.\footnote{Otherwise, restrict to the coordinates where the diagonal of $G_s$ is positive: since $G_s \succeq 0$, every row of $G_s$ with a zero diagonal entry is identically zero, and such coordinates contribute to neither side of the claimed inequality.} Define $\widetilde{G}_s := \diag(G_s)^{- 1/2} G_s \diag(G_s)^{- 1/2}$. 
Then since $\widetilde{G}_s$ has unit diagonal and at most $r_s h$ rows, we have
\[
\Tr(\widetilde{G}_s) \leq r_s h .
\]
Let  $W_{s,G} \subset \R^{m_s}$ be the span of all eigenvectors of $\widetilde{G}_s$ with eigenvalues at most $\alpha_s$. 
Then the above bound implies that $\dim(W_{s,G}) \geq m_s - r_s \alpha_s^{-1} h$, as there can be  at most $r_s \alpha_s^{-1} h$ eigenvalues exceeding $\alpha_s$. 
Define $W'_{s,G} := \diag(G_s)^{-1/2} W_{s,G}$, which also has $\dim(W'_{s,G}) \geq m_s - r_s \alpha_s^{-1} h$. 
Then for any vector $u \in W'_{s,G}$, we have
\begin{align*}
u^\top G_s u & = (\diag(G_s)^{1/2} u)^\top \widetilde{G}_s (\diag(G_s)^{1/2} u )\\
& \leq \alpha_s \cdot (\diag(G_s)^{1/2} u)^\top \diag(G_s)^{1/2} u = \alpha_s u^\top \diag(G_s) u ,
\end{align*}
where the inequality above follows as $\diag(G_s)^{1/2} u \in W_{s,G}$ and all eigenvalues of $\widetilde{G}_s$ in the subspace  $W_{s,G}$ are at most $\alpha_s$. 

We view $E_s: \mathbb{R}^h \rightarrow \mathbb{R}^{m_s}$
as a linear map with kernel $\mathsf{Ker}(E_s)$. 
Note that any vector $u \in \mathsf{Ker}(E_s)$ clearly satisfies the statement of the claim, since both sides of the inequality are $0$. 
Any vector $u \in \mathbb{R}^h$ such that $E_s u \in W'_{s,G}$ also satisfies the claim statement. 
Therefore, the subspace 
\[ W_s := E_s^{-1}(E_s(\R^h) \cap W_{s,G}') \]
satisfies $\dim(W_s) \geq (1 - r_s \alpha_s^{-1}) h$ and the statement of the claim. 
\end{proof}

\section{Proof of the Improved Koml\'os Bound}
\label{sec:komlos_proof}

In this section we prove \Cref{thm:main}, restated below. 

\maintheorem*

We will assume that $n \geq 100$, as otherwise the result holds trivially. 
Also, it suffices to bound only the one-sided discrepancy $\max_{i \in [m]} \langle a_i,x \rangle = O((\log n)^{1/4} (\log \log n)^{7/4})$, as we can stack the matrix $-A$ below the input matrix $A$ (and rescale by $1/\sqrt{2}$ so that the column $\ell_2$ norm is at most $1$). We may also assume that $m\le n^2$: any row with $\|a_i\|_2\le n^{-1/2}$ satisfies $|\langle a_i,x\rangle|\le \|a_i\|_2\|x\|_2\le 1$ for every $x\in\{-1,1\}^n$ and may be discarded, and since the columns of $A$ have $\ell_2$ norm at most one, fewer than $n^{2}$ rows have $\ell_2$ norm exceeding $n^{-1/2}$.

The proof will be based on constructing an It\^o process $(x_t)_{0 \leq t \leq T} \subset [-1,1]^n$ with $x_0 = 0$ and increment $d x_t = \Sigma_t d W_t$, for a suitably chosen diffusion matrix $\Sigma_t \in \R^{n \times n}$ to ensure that 
$x_T\in \{-1,1\}^n$ and it satisfies the discrepancy bound in \Cref{thm:main} with high probability.

The process $x_t$ will also satisfy the {\em sticky} property that $dx_t(j) = 0$ whenever $x_t(j) \in \{\pm 1\}$, so that each coordinate remains unchanged once it reaches $\pm 1$. We denote by $V_t := \{j \in [n]: |x_t(j)| < 1\}$ the set of {\em alive} coordinates at time $t$, and coordinates in $[n] \setminus V_t$ will be called {\em dead} or {\em frozen}. 
We call $x_t$ the {\em fractional} coloring at time $t$. 

\subsection{Overview}

Before describing the process formally, we first sketch the main ideas.
For simplicity, consider the case where each entry of $A$ has magnitude $0$ or $1/\sqrt{k}$, with at most $k$ non-zero entries per column. We will show a discrepancy bound of $b \approx (\log \log n)^{1/2}$ whenever $k \geq \log^3 n$, which is substantially stronger than the bound of \Cref{thm:main}.

The basic idea is to track the discrepancy of each row over time as the coloring $x_t$ evolves, and prevent it from increasing beyond our target $O(b)$.
A key tool, which goes back to \cite{BDG19}, is the following:
At any time $t$, given any (say) $|V_t|/2$ bad directions, there exists an increment $dx_t$ that is orthogonal to all these directions, and yet behaves like a random direction (more precisely, it is $O(1)$-spectrally independent). Let us see how we can use it.

Fix some time $t$. Call a row $a_i$ {\em large} if its support is at least $10k$ on $V_t$, and {\em small} if it is at most $k/\log{n}$. Otherwise, call it {\em medium}. As each column in $V_t$ has support $\leq k$, there can be at most $|V_t|/10$ large rows, and at most $|V_t|\log n$ medium rows at any time $t$.
By the tool above, we can easily control large and small rows --- at each $t$, simply choose $dx_t$ to be orthogonal to all the (at most $|V_t|/10$) large rows. This ensures that each row incurs zero discrepancy while it is large. Once a row is small, by \Cref{prop:spec-indep-subg-disc}, the random-like property of $dx_t$ ensures that it incurs only $O(1)$ discrepancy with high probability.

So it suffices to focus on medium rows. 
Call a row $a_i$ {\em bad} if its discrepancy ever reaches $b$.
By the random-like property of $dx_t$, a typical row $a_i$ becomes bad with probability  $\exp(-\Theta(b^2))$ which is $(\log n)^{-O(1)}$  by our choice of $b$. So, unfortunately, once $|V_t| \ll m (\log n)^{-O(1)}$, one cannot hope to choose $dx_t$ orthogonal to all bad rows.

Interestingly, one can still ensure that with high probability, at most $|V_t|/10$ rows can be {\em both} bad and medium at all times $t$, irrespective of the subset $V_t$ of alive coordinates. This would give the result, as now
we can also choose $dx_t$ to be orthogonal to all such rows, at each time $t$, thereby ensuring that the discrepancy incurred by medium rows is at most $b$.
To this end, we will crucially use the affine spectral independence property, which ensures that the row discrepancies evolve almost independently. We describe this next.

\smallskip
{\bf Approximate Independence.} 
Suppose first that the row discrepancies evolve (completely) independently. 
Fix some column $j \in [n]$, and consider the at most $k$ rows in its support, i.e.,~$a_i$ with $a_i(j)\neq 0$. As the row discrepancies evolve independently, by Chernoff bounds, at most $k\exp(-\Omega(b^2)) +O(\log n) \leq k/(10\log n)$ of these rows will become bad, with high probability. By a union bound, with high probability this holds simultaneously for every column $j \in [n]$; condition on this event. Then at any time $t$, as each medium row has support $\geq k/\log n$ in $V_t$, a simple averaging argument gives that there can be at most $|V_t|/10$ bad medium rows.

While full independence is too much to expect, one can still ensure that the discrepancy increments for medium rows are almost independent with parameter $\alpha=O(\log n)$. 
This follows from \Cref{thm:sub-isotropic-SDP_Multi-Class}, using the fact that there are only $O(|V_t|\log n)$ medium rows at any time $t$.  
This only increases the additive term to $O(\alpha \log n)$ in the bound above on bad rows in a column, and the same argument works provided $k\geq \log^3n$.

\smallskip
{\bf The General Case.}
The argument for the general Koml\'{o}s problem is similar, except that one needs to handle multiple ``scales" of $k$, and the discrepancy bound in the argument above degrades as $k \ll \log^3 n$.
In particular, we decompose $A$ into $O(\log \log n)$ geometrically decreasing ``scales'' of $k$ below a fixed power of $\log n$, together with a single scale collecting all larger $k$, and handle each scale separately. 
These different scales make the argument somewhat more technical, as the parameters must be set differently for each scale.

\subsection{Multi-Scale Decomposition}

We start with a multi-scale decomposition of the input matrix $A$.

Set $P:= 1+\lceil 5\log_2\log n\rceil$.
For each row $a_i$ and each scale $p\in[P-1]$, define $a_i^{(p)}\in\mathbb R^n$ by
\[
a_i^{(p)}(j):=
\begin{cases}
a_i(j), & \text{if } |a_i(j)|\in(2^{-p},2^{-p+1}],\\
0, & \text{otherwise.}
\end{cases}
\]
The final scale $P$ will behave differently, and we define
\[
a_i^{(P)}(j):=
\begin{cases}
a_i(j), & \text{if } |a_i(j)|\le 2^{-(P-1)},\\
0, & \text{otherwise.}
\end{cases}
\]
Note that $2^{-(P-1)}\le \log^{-5}n$, so every non-zero entry of $A^{(P)}$ has magnitude at most $\log^{-5}n$.
Let $A^{(p)} \in \R^{m \times n}$ denote the matrix with rows $a_i^{(p)}$. Then we have the decomposition $A=\sum_{p=1}^P A^{(p)}$.
Note that for each $p \in [P-1]$, each non-zero entry of $A^{(p)}$ has magnitude in $(2^{-p},2^{-p+1}]$, and thus each column of $A^{(p)}$ has at most $2^{2p}$ non-zero entries.

\smallskip
\noindent\textbf{Target Discrepancy.}
We fix a common target discrepancy bound for all scales \[b := C (\log n)^{1/4} (\log \log n)^{3/4}.\]
We will show that with high probability, $\|A^{(p)} x_T \|_\infty \leq O(b)$ for each $p$. This will imply \Cref{thm:main}, as
$\|A x_T\|_\infty \leq \sum_{p \in [P]}  \|A^{(p)} x_T\|_\infty = O(bP) = O\big((\log n)^{1/4}(\log\log n)^{7/4}\big)$.

\smallskip
\noindent \textbf{Row Classification.} 
For each scale $p \in [P]$, define the threshold
\begin{align} \label{eq:size-threshold_Komlos}
\mu_p:=\max\!\Big\{\frac{4 b}{2^p},\frac{b^2}{\log n}\Big\}. 
\end{align}

At each time $0 \leq t \leq T$, recall that $V_t \subseteq [n]$ is the set of alive columns. 
We classify the rows based on squared $\ell_2$ mass  in $V_t$ as follows. 

\begin{definition}[Large, medium, and small rows]
At each time $0 \leq t \leq T$, a row $i \in [m]$ is called

(i) \emph{scale-$p$ large} if $\sum_{j\in V_t}a_i^{(p)}(j)^2>10$;

(ii) \emph{scale-$p$ small} if $\sum_{j\in V_t}a_i^{(p)}(j)^2\le \min(\mu_p, 10)$, and

(iii) \emph{scale-$p$ medium} if $\sum_{j\in V_t}a_i^{(p)}(j)^2\in (\mu_p, 10]$. 
\end{definition}

Notice that the choice of $\mu_p$ in \eqref{eq:size-threshold_Komlos} ensures that every scale-$p$ small row $a_i$ either has $\ell_1$-norm $O(b)$ or has $\ell_2$-norm $O(b/\sqrt{\log n})$, restricted to the coordinates in $V_t$.

Let $\mathcal{L}^{(p)}_t, \mathcal{M}^{(p)}_t \subseteq [m]$ denote the set of scale-$p$ large and medium rows at time $t$.
As each column has $\ell_2$ norm at most $1$, the total squared $\ell_2$ mass $\sum_{j \in V_t} \sum_{i \in [m]} a_i(j)^2$ in the columns $V_t$ is at most $|V_t|$. Thus
an averaging argument over rows gives that
\begin{align}
\label{claim:M_t_bound-Komlos}
 |\mathcal{M}_t^{(p)}| \leq |V_t|/\mu_p \text{ for each $p$, and that } \sum_p  |\mathcal{L}^{(p)}_t| \leq |V_t|/10.
\end{align}

{\bf Medium rows.}
As discussed previously, bounding the discrepancy of rows when they are large or small will be easy, and our focus will thus be on controlling the discrepancies of medium rows.
To control the discrepancy of a scale-$p$ medium row $i$, we will control its regularized discrepancy 
\begin{align} \label{eq:reg-disc-med-Komlos}
Y_t^{(p)}(i)
:=
\bigl\langle a_i^{(p)},x_t\bigr\rangle
+
\beta\sum_{j=1}^n \bigl(a_i^{(p)}(j)\bigr)^2\bigl(1-x_t(j)^2\bigr) ,
\end{align}
where $\beta := b/10$. 

As $1-x_t(j)^2=0$ for dead $j\in [n]\setminus V_t$, and $\sum_{j\in V_t} a_i^{(p)}(j)^2 \leq 10$ for medium rows,  one always has 
\begin{align} \label{eq:reg-disc-approx-Komlos}
\bigl\langle a_i^{(p)},x_t\bigr\rangle \leq Y_t^{(p)}(i) \leq \bigl\langle a_i^{(p)},x_t\bigr\rangle + b .
\end{align}
We call a scale-$p$ medium row $i$ \emph{dangerous} at time $t$ if
\[Y_t^{(p)}(i) \geq 2b.\]
Note that we control the regularized discrepancy rather than $\langle a_i^{(p)},x_t\rangle$ itself: it is $Y_t^{(p)}$ that has the negative-drift structure of \Cref{lem:reg-disc-increment}, and by \eqref{eq:reg-disc-approx-Komlos} a bound on $Y_t^{(p)}(i)$ implies the same bound on the discrepancy.
Finally, define the matrix $E_t^{(p)} \in \R^{\mathcal{M}_t^{(p)} \times n}$, corresponding to \eqref{eq:E_t-reg-disc}, with rows 
\[
E_{t}^{(p)} (i,\cdot) := a_i^{(p)}-2\beta\,(a_i^{(p)\odot 2}\odot x_t) .
\]

\subsection{The It\^o Process} \label{subsec:ito_Komlos} 
We now formally define the It\^o process $(x_t)_{0 \leq t \leq T} \subset [-1,1]^n$. 
The process will start at
 $x_0 = 0$ and evolve with increments $d x_t = \Sigma_t d W_t$, based on a diffusion matrix $\Sigma_t \in \R^{n \times n}$. 
To describe $\Sigma_t$, we first define a subspace $H_t$ of vectors to which $dx_t$ will be orthogonal.

\begin{definition}[Blocking subspace] \label{def:blocking_Komlos}
For each time $0 \leq t \leq T$, let the blocking subspace $H_t \subset \R^{n}$ be the linear span of the following collection of vectors:

 \,\,\,   (i) $e_j$ for every dead column $j \in [n] \setminus V_t$, where $e_j$ is the $j$th standard basis vector;
    
   \,\,\,   (ii) $a_i^{(p)}$ for every scale-$p$ large row $i$; 

   \,\,\,   (iii) $E_t^{(p)}(i, \cdot)$ for every dangerous  scale-$p$ row $i$;

   \,\,\,   (iv) The vector $x_t$. 
\end{definition}

To avoid cumbersome conditions 
in the description of $\Sigma_t$, we will assume that $|V_t|\geq 10$ henceforth. Once $|V_t|$ reaches $10$, we stop the process and simply round each alive variable to $\pm 1$, which increases the discrepancy by at most $10$.

\begin{definition}[Diffusion matrix $\Sigma_t$]  \label{def:Sigma_t-Komlos}
 Let $U_t \in \R^{n \times n}$ be a PSD matrix satisfying:

\,\,\,  (i) $\langle vv^\top, U_t \rangle = 0$ for all $v \in H_t$, where the blocking subspace $H_t$ is given by \Cref{def:blocking_Komlos}, 

\,\,\,   (ii) $U_t(j,j) \leq 1$ for all $j \in V_t$, 

\,\,\,  (iii) $\Tr(U_t) \geq |V_t|/6$, 

\,\,\,  (iv) $U_t \preceq 6 \, \diag(U_t)$, and 

\,\,\,  (v) $E_t^{(p)} U_t (E_t^{(p)})^\top \preceq (6P/\mu_p) \, \diag\big(E_t^{(p)} U_t (E_t^{(p)})^\top \big)$ for every $p \in [P]$. 

Set $\Sigma_t := U_t^{1/2}/\sqrt{\Tr(U_t)}$.
\end{definition}
Constraint (i) ensures that the increment $dx_t$ is orthogonal to $H_t$. In particular, $dx_t(j)=0$ for all dead columns $j \in [n]\setminus V_t$, so the process is sticky as promised. Constraints (iv) and (v) ensure the spectral independence and affine spectral independence properties as in \Cref{prop:weak_ind_reg-disc}. Constraints (ii) and (iii) rule out the all-zero solution for $U_t$.

If $\Sigma_t$ does not exist at any time $t$, we abort the process and declare failure. Provided that $\Sigma_t$ exists for all $t$, $\Sigma_t$ is clearly progressively measurable, and hence $x_t$ is a It\^o process.

The following shows that $\Sigma_t$ always exists if the dimension of $H_t$ is not too large. 

\begin{lemma}[Existence of $\Sigma_t$] \label{lem:existence-Sigma_t}
If the blocking subspace $H_t \subset \R^n$ in \Cref{def:blocking_Komlos} satisfies $\dim(H_t) \leq n - 2|V_t|/3$, then $\Sigma_t \in \R^{n\times n}$ as in \Cref{def:Sigma_t-Komlos} exists. 
In particular, $\Sigma_t$ exists if the total number of dangerous rows (over all scales) is at most $|V_t|/10$. 
\end{lemma}

\begin{proof}
As $e_j \in H_t$ for every dead column $j \in [n] \setminus V_t$, condition (i) in \Cref{def:Sigma_t-Komlos} forces that $U_t(j,j) = 0$ for all $j \in [n] \setminus V_t$. Thus, as $U_t$ is PSD, it can only have non-zero entries in $V_t \times V_t$. 

As $\R^{[n]\setminus V_t} \subseteq H_t$, the subspace $H_t \cap \R^{V_t}$ has dimension at most $ \dim(H_t) - (n-|V_t|) \leq |V_t|/3$. 
By \eqref{claim:M_t_bound-Komlos}, the number of rows of each $E_t^{(p)}$ is at most $|V_t|/\mu_p$.
Therefore, we can apply \Cref{thm:sub-isotropic-SDP_Multi-Class} to the subspace $H = H_t \cap \R^{V_t}$ with $h= |V_t|$ and parameters $\delta = 1/3$ and $r_p \leq 1/\mu_p$, and $\kappa = 1/6$, $\eta = 6$, $\alpha_p = (6P/\mu_p)$ (corresponding to constraints (iii)-(v) in \Cref{def:Sigma_t-Komlos}).
The existence of $U_t$, and hence $\Sigma_t$, then follows as 
\[
\eta^{-1} + \kappa + \sum_{p \in [P]} r_p \alpha_p^{-1} \leq 
\frac{1}{2} \leq 1 - \delta.
\]
For the second part, let $\mathcal{L}_t$ denote the set of large rows (over all scales), and $\mathcal{D}_t$ the set of dangerous rows at time $t$.  As $|\mathcal{L}_t| \leq |V_t|/10$ by \eqref{claim:M_t_bound-Komlos} and we have assumed that $|V_t| \geq 10$, 
 \[\dim(H_t) \leq (n-|V_t|) + |\mathcal{L}_t| + |\mathcal{D}_t| + 1 \leq n - 0.9|V_t| + |\mathcal{D}_t|, \]
 and thus  $\dim(H_t) \le n-2|V_t|/3$ provided that $|\mathcal{D}_t| \leq |V_t|/10$. 
\end{proof}

\subsection{Analysis}
We now prove \Cref{thm:main}. 
We first show that if the process does not abort, i.e.,~$\Sigma_t$ exists at all times $t$, then the discrepancy bound in \Cref{thm:main} holds with high probability. 
Next, we show that with high probability, the process does not abort.

For each row $i\in[m]$ and scale $p\in[P]$, define the stopping times
\begin{align} \label{eq:stopping-times-Komlos}
\tau_{i,\mathsf{med}}^{(p)} := \inf\{t\ge 0: \text{row $i$ not scale-$p$ large}\},\quad
\tau_{i,\mathsf{small}}^{(p)} := \inf\{t\ge 0:\ \text{row $i$ scale-$p$ small}\},
\end{align}
with the convention $\inf\emptyset := T$. Since $V_t$ shrinks over time, $\sum_{j\in V_t}a_i^{(p)}(j)^2$ is non-increasing, 
so row $i$ is scale-$p$ large precisely on $[0,\tau_{i,\mathsf{med}}^{(p)})$, medium on $[\tau_{i,\mathsf{med}}^{(p)},\tau_{i,\mathsf{small}}^{(p)})$, and small on $[\tau_{i,\mathsf{small}}^{(p)},T]$.

\begin{lemma}[Discrepancy bound] \label{lem:disc-given-Sigma}
    Suppose $\Sigma_t$ exists for all $0\le t\le T$. Then $T\le n$ almost surely, and with probability at least $1-1/n$, the output coloring $x\in\{-1,1\}^n$ satisfies $\|Ax\|_\infty = O(bP)$.
\end{lemma}

\begin{proof}

We first show that the process reaches $|V_t|\le 10$ by time $n$, i.e., $T\le n$ almost surely.
Indeed, applying It\^o's formula to $\|x_t\|_2^2$, we have
\[
d\|x_t\|_2^2
=
2\langle x_t,dx_t\rangle + \Tr(d[x, x]_t)
=
\Tr(\Sigma_t \Sigma_t^\top)\,dt = dt,
\]
as $\langle x_t,dx_t\rangle = 0$ since $x_t \in H_t$ and $\Tr(\Sigma_t \Sigma_t^\top) =1$ by \Cref{def:Sigma_t-Komlos}. 

Thus $\|x_t\|_2^2=t$ for all $0\le t\le T$.  
As $x_t \in [-1,1]^n$ and hence $\|x_t\|_2^2 \leq n$, it follows that $T\le n$.  

\smallskip
We now bound the discrepancy.
Fix scale \(p \in [P]\) and row \(i \in [m]\). We show that $\sup_{0 \leq t \leq T} \langle a_i^{(p)},x_t\rangle = O(b)$ with
probability $1-1/n^{4}$. 
A union bound over all $p,i$ then gives
$\|Ax_T\|_\infty \leq \sum_{p \in [P]}  \|A^{(p)} x_T\|_\infty = O(bP)$ with probability at least $1 - mP/n^{4} \ge 1-1/n$, using $m\le n^2$.

To this end,
let us track the discrepancy incurred by $a_i^{(p)}$ over time.
Recall the stopping times \eqref{eq:stopping-times-Komlos}.
While row $i$ is scale-$p$ large, i.e.,~$t < \tau_{i,\mathsf{med}}^{(p)}$, we have $a_i^{(p)} \in H_t$ and thus it incurs no discrepancy.

While row $i$ is scale-$p$ medium, 
we claim that
its regularized discrepancy $Y_t^{(p)}(i)$  never exceeds $2b$. 
Indeed,  $Y_{t}^{(p)}(i) \leq b$ initially at $t=\tau_{i,\mathsf{med}}^{(p)}$ (as $\beta = b/10$). Later, if $Y_t^{(p)}(i)$ reaches $2b$, the row is dangerous and the vector $E_t^{(p)}(i, \cdot) \in H_t$, which ensures that $d Y_t^{(p)}(i) \leq  0$ (deterministically) by \Cref{lem:reg-disc-increment}.  
As $\langle a_i^{(p)},x_t\rangle \le Y_t^{(p)}(i)$ for all $t$, this implies that its discrepancy never exceeds $2b$. 

When row $i$ becomes scale-$p$ small, we consider two cases depending on $p$: whether $\mu_p = 4b/2^p$ or $b^2/\log n$ in \eqref{eq:size-threshold_Komlos}.
The case $\mu_p = 4b/2^p$ holds for $p \leq \log ((4 \log n)/b) < P$; for such $p$ (in particular $p\le P-1$), all non-zero entries of $a_i^{(p)}$ have absolute values in  $(2^{-p}, 2^{-p+1}]$. 
So \[\sum_{j \in V_t} |a_i^{(p)}(j)| \leq 2^p \sum_{j \in V_t}  a_i^{(p)}(j)^2 \leq 2^p \mu_p = O(b),\] and
thus row $i$ can incur at most $O(b)$ discrepancy.

When $\mu_p = b^2/\log n$, we use the spectral independence constraint (iv) of \Cref{def:Sigma_t-Komlos}. Applying \Cref{prop:spec-indep-subg-disc} on the interval $[\tau_{i,\mathsf{small}}^{(p)}, T]$ to the vector $a_i^{(p)}$ restricted to $V_{\tau_{i,\mathsf{small}}^{(p)}}$ (which has squared $\ell_2$ norm at most $\mu_p$), with $\eta = 6$ and $\gamma = n^{-4}$, we get that with probability at least $1 - n^{-4}$, the discrepancy incurred during this interval is $O((\mu_p \log n)^{1/2}) = O(b)$.
\end{proof}

We now show that with high probability, $\Sigma_t$ exists for all times $t$. To this end, by \Cref{lem:existence-Sigma_t} it suffices to show the following.

\begin{lemma}[Few dangerous rows] 
\label{lem:key-lemma-Komlos}
For the process $x_t$, with probability at least $1-1/n$,
the total number of dangerous rows across all scales $p \in [P]$ is at most $|V_t|/10$,
 for every $0 \leq t \leq T$. 
 \end{lemma}

\begin{proof}
We will show that, with high probability, for all $0 \leq t \leq T$, the number of dangerous scale-$p$ rows is at most $|V_t|/(10P)$ for each scale $p \in [P]$.

The idea is the following. By definition, each scale-$p$ medium row has at least $\mu_p$ squared-$\ell_2$-mass $\sum_{j\in V_t}a_i^{(p)}(j)^2 $ on $V_t$. So to bound the number of dangerous rows at time $t$ by $|V_t|/(10P)$, it suffices to upper bound this total mass over all scale-$p$ dangerous rows by $(|V_t|/10P) \mu_p$.
To handle all possible subsets $V_t$, we will show that,  
with high probability,
this mass is at most $\mu_p/10P$ for every column $j\in [n]$, at all times $t$.
To show the latter, we will carefully use the affine spectral independence constraints.

We now give the details.
We first consider scales $ p \leq P-1$. Scale $P$ needs a separate argument as the non-zero entries there can vary widely in magnitude.

\smallskip
{\bf Scales $p \leq P-1$:} First, if $p< \log_2(2b/5)$, the result holds vacuously as $\mu_p = 4b/2^p >  10$, and thus there are no medium scale-$p$ rows by definition. 
So we assume that $p\geq  \log_2(2b/5)$.

Fix any scale $\log_2(2b/5) \leq p \leq P-1$ and any column $j \in [n]$. Let \[N_j^{(p)} := \{i \in [m]: a_i^{(p)}(j) \neq 0\}\]
be the set of rows with non-zero scale-$p$ entries in column $j$. 
Then $|a_i^{(p)}(j)| \in (2^{-p},2^{-p+1}]$ for each such entry, and $|N_j^{(p)}| \leq 2^{2p}$. 
Let $N_{j,{\mathsf{bad}}}^{(p)}$ denote the subset of rows of $N_j^{(p)}$ that ever become dangerous during the process.
We will show that $|N_{j,{\mathsf{bad}}}^{(p)}| \leq 2^{2p} (\mu_p/40P)$, with probability at least $1-n^{-4}$, and thus the squared $\ell_2$ mass of such entries is at most $\mu_p/10P$, as desired.

To do this, we would like to apply the decoupling bound of \Cref{cor:decoupling_reg-disc} to the regularized discrepancies $Y_t^{(p)}(i)$, as in \eqref{eq:reg-disc-med-Komlos}, of the scale-$p$ medium rows.
However, there is a technical issue that the set $\mathcal{M}_t^{(p)}$ changes over time. 
To handle this, we first extend the definition of $Y_t^{(p)}$ to all rows, and introduce a modified process $Z_t^{(p)}$ with entries $Z_t^{(p)}(i)$ defined for all $i \in [m]$ as follows: 
\begin{align} \label{eq:modified_reg-disc-Komlos}
Z_t^{(p)}(i) := 
\begin{cases}
b \quad &\text{when $0\leq t < \tau_{i,\mathsf{med}}^{(p)}$,} \\
Y_t^{(p)}(i) + (b - Y_{\tau_{i,\mathsf{med}}^{(p)}}^{(p)}(i) ) \quad & \text{when $ \tau_{i,\mathsf{med}}^{(p)} \leq t < \tau_{i,\mathsf{small}}^{(p)}$,} \\
Y_{{\tau_{i,\mathsf{small}}^{(p)}}}^{(p)}(i) + (b - Y_{\tau_{i,\mathsf{med}}^{(p)}}^{(p)}(i)) \quad &\text{otherwise,}
\end{cases}
\end{align}
where $\tau_{i,\mathsf{med}}^{(p)}$ and $\tau_{i,\mathsf{small}}^{(p)}$ are as in \eqref{eq:stopping-times-Komlos}.

Note that $Z_t^{(p)}(i)$ has $0$ increment while row $i$ is scale-$p$ large or small, and has increment $dZ_t^{(p)}(i) = d Y_t^{(p)}(i)$ when it is scale-$p$ medium. Moreover, as $Y_{t}^{(p)}(i) \leq b = Z_{t}^{(p)}(i)$ at $t=\tau_{i,\mathsf{med}}^{(p)}$, we have that $Z_{t}^{(p)}(i) \geq Y_{t}^{(p)}(i) $ for all $\tau_{i,\mathsf{med}}^{(p)} \leq t < \tau_{i,\mathsf{small}}^{(p)}$.
In particular, $Z_t^{(p)}(i) - Z_0^{(p)}(i) \geq b$ whenever row $i$ is dangerous at time $t$. Thus we can upper bound $|N_{j,{\mathsf{bad}}}^{(p)}|$:
\begin{align}
\label{eq:relate-nj-zj}
    |N_{j,{\mathsf{bad}}}^{(p)}| \leq \Big| \big\{i \in N_j^{(p)} : \sup_{0\leq t \leq T} \big(Z_t^{(p)}(i) - Z_0^{(p)}(i)\big) \geq b \big\} \Big|.
\end{align} 
We will use the decoupling bound in \Cref{cor:decoupling_reg-disc} to bound the right side.

Note that every scale-$p$ row satisfies the hypothesis $\beta\|a_i^{(p)}\|_\infty\le 1/2$ of \Cref{prop:weak_ind_reg-disc}, since $\beta=b/10$ and $\|a_i^{(p)}\|_\infty\le 2^{-p+1}\le 5/b$ in the present range of $p \geq \log_2(2b/5)$. 
As the diffusion matrix satisfies constraints (iv) and (v) of \Cref{def:Sigma_t-Komlos} with $\eta=6$ and $\alpha = 6P/\mu_p$, 
\Cref{prop:weak_ind_reg-disc} therefore applies to the coordinates that are currently scale-$p$ medium, for which $dZ_t^{(p)}(i)=dY_t^{(p)}(i)$. The remaining coordinates have $dZ_t^{(p)}(i)=0$ and trivially satisfy both conditions of \Cref{def:alpha-theta-independence}. 
Hence $Z_t^{(p)}$ is an $(\alpha,\beta/4\eta)$-independent It\^o process. The same holds for its restriction $Z_{j,t}^{(p)}$ to the coordinates in $N_j^{(p)}$, as both defining conditions pass to principal submatrices of the covariance.
 
To bound the right side in \eqref{eq:relate-nj-zj}, we apply \Cref{cor:decoupling_reg-disc} to  $Z_{j,t}^{(p)}$ with threshold $B= b$ and parameters $\gamma = n^{-4}$ and $\lambda = (\log\log n)/b$. This satisfies the required condition $\lambda \le \beta/(4\eta)$: indeed $\lambda=(\log\log n)/b\le b/240=\beta/(4\eta)$, since $b^{2}\gg\log\log n$. 
Plugging in these parameters together with $\alpha = (6P)/\mu_p$, and using
 \eqref{eq:relate-nj-zj}, we have 
 with probability at least $1 - n^{-4}$, 
\begin{align} \label{eq:small_scale_bound}
|N_{j,\mathsf{bad}}^{(p)}|  \le
|N_j^{(p)}| e^{-\lambda B} +
\frac{4 \eta \alpha\lambda}{\beta}\log(n^{4})  \le \frac{|N_j^{(p)}|}{\log n}+
\frac{c P \log n \log \log n}{b^2 \mu_p},
\end{align}
where $c$ is a fixed constant independent of the constant $C$ in the definition of $b$.

We claim that the bound in \eqref{eq:small_scale_bound} is at most $2^{2p} \mu_p/40P$. 
Indeed, as $|N_{j}^{(p)}| \leq 2^{2p}$,
the first term is at most $2^{2p}/\log n \leq 2^{2p} \mu_p/80P$, as $\mu_p \geq b^2/\log n = \omega(P/\log n)$.

For the second term, using $\mu_p \geq (4b)/2^p$ we have that 
\[ \frac{c P \log n \log \log n}{b^2 \mu_p^2} \leq  \frac{c 2^{2p}} {16} \cdot \frac{P \log n \log \log n}{b^4} \ll 2^{2p} \frac{1}{80P},
\]
where the final inequality follows as $P = O(\log \log n)$ and by our choice of $b = C \log^{1/4} n (\log \log n)^{3/4}$ with $C$ large enough.

\smallskip
\noindent\textbf{The scale $P$.} 
As before, for each column $j$, we will bound the total squared $\ell_2$-mass that column $j$ carries in dangerous scale-$P$ rows by $\mu_P/10P$.

Fix a column \(j\). Partition the non-zero entries of \(A^{(P)}\) into subscales
$q \geq  P$ where subscale \(q\) consists of entries with magnitude in
\((2^{-q},2^{-q+1}]\). Let \(k_{j,q}\) denote the number of scale-\(P\), subscale-\(q\) entries
in column \(j\). As column \(j\) has \(\ell_2\)-norm at most \(1\),
we have 
$\sum_q 2^{-2q} k_{j,q} \le 1.$

As before, consider the scale-$P$ regularized discrepancy process $Z^{(P)}_t$, and for a subscale $q$ let $Z^{(P)}_{j,q,t}$ denote its restriction to the rows with a subscale-$q$ entry in column $j$. By constraints (iv) and (v) of \Cref{def:Sigma_t-Komlos} for $p=P$, together with \Cref{prop:weak_ind_reg-disc} (whose hypothesis $\beta\|a_i^{(P)}\|_\infty\le 1/2$ holds as $\|a_i^{(P)}\|_\infty\le \log^{-5}n$), the process $Z^{(P)}_{j,q,t}$ is $(\alpha,\beta/4\eta)$-independent with $\alpha=6P/\mu_P$ and $\eta=6$. Applying the decoupling bound \Cref{cor:decoupling_reg-disc} with threshold $B=b$, $\lambda=(2\log\log n)/b$ and $\gamma=n^{-4}$ (as before, $\lambda\le\beta/4\eta$), we get that with probability at least $1-n^{-4}$, the number of dangerous subscale-$q$ rows in column $j$ is at most
\[
\begin{aligned}
&k_{j,q} \log^{-2} n + \frac{4 \eta \alpha\lambda}{\beta}\log(n^{4}) \le\ k_{j,q} \log^{-2} n +
O(\log n).
\end{aligned}
\]
Here we used that

\[\frac{4\eta\alpha\lambda}{\beta}\,\log(n^{4}) = O\!\Big(\frac{P\,\log\log n\,\log n}{\mu_P\, b^{2}}\Big) \le O\!\Big(\frac{P\,\log\log n\,\log^{2} n}{b^{4}}\Big) = O\!\Big(\frac{\log n}{\log\log n}\Big) \le O(\log n),\]
using $\mu_P\ge b^{2}/\log n$, $b^{4}=C^{4}\log n\,(\log\log n)^{3}$ and $P=O(\log\log n)$. 
 Since all entries in  subscale $q$ have magnitude at most \(2^{-q+1}\), it follows that the
total squared \(\ell_2\)-mass of dangerous scale-\(P\) entries in column \(j\) is at most
\[
\begin{aligned}
\sum_{q\geq P} 2^{-2q+2}\Bigl(k_{j,q}\log^{-2} n + O(\log n)\Bigr)
& \le O(\log^{-2} n) + O(2^{-2P}\log n) = O(\log^{-2} n) \ll
\frac{\mu_P}{10P},
\end{aligned}
\]
since $\mu_P\ge b^{2}/\log n = C^{2}(\log\log n)^{3/2}/\sqrt{\log n}$ and $P=\Theta(\log\log n)$.

\smallskip
\noindent \textbf{Wrapping Up.} Finally, we combine the column-wise bounds. Condition on the event that all the above bounds hold for every column $j\in[n]$, 
which happens with probability at least $1-1/n$ by a union bound. Fix a time $t$ and a scale $p$, and let $\mathcal{D}_t^{(p)}$ denote the set of scale-$p$ dangerous rows at time $t$. Every row in $\mathcal{D}_t^{(p)}$ is scale-$p$ medium at time $t$, and hence carries squared $\ell_2$ mass more than $\mu_p$ on $V_t$; moreover, it has become dangerous by time $t$, and thus contributes to the dangerous mass of every column in its support. Consequently,
\[
|\mathcal{D}_t^{(p)}|\,\mu_p \;\le\; \sum_{j\in V_t}\ \sum_{i\in \mathcal{D}_t^{(p)}} a_i^{(p)}(j)^2 \;\le\; \sum_{j\in V_t}\frac{\mu_p}{10P} \;=\; |V_t|\,\frac{\mu_p}{10P},
\]
so $|\mathcal{D}_t^{(p)}|\le |V_t|/(10P)$. Summing over $p\in[P]$ completes the proof.
\end{proof}

\begin{proof}[Proof of \Cref{thm:main}]
By \Cref{lem:key-lemma-Komlos}, with probability at least $1-1/n$ the number of dangerous rows is at most $|V_t|/10$ for all $0\le t\le T$, and on this event $\Sigma_t$ exists for all $t$ by \Cref{lem:existence-Sigma_t}. By \Cref{lem:disc-given-Sigma}, on this event the output coloring $x\in\{-1,1\}^n$ satisfies, with probability at least $1-2/n$, $\|Ax\|_\infty = O(bP) = O\big((\log n)^{1/4}(\log\log n)^{7/4}\big)$. This proves the theorem. 
\end{proof}

\bibliographystyle{alpha}
\bibliography{bib.bib}

@article{AT26,
  title={Online Beck--Fiala Down to Logarithmic Sparsity},
  author={Altschuler, Dylan J and Tikhomirov, Konstantin},
  journal={arXiv preprint arXiv:2607.14238},
  year={2026}
}

@inproceedings{BJ26,
  title={Decoupling via affine spectral-independence: Beck-Fiala and Koml{\'o}s bounds beyond Banaszczyk},
  author={Bansal, Nikhil and Jiang, Haotian},
  booktitle={Proceedings of the 58th Annual ACM Symposium on Theory of Computing},
  pages={432--442},
  year={2026}
}

@book{BV04book,
  title={Convex optimization},
  author={Boyd, Stephen P and Vandenberghe, Lieven},
  year={2004},
  publisher={Cambridge university press}
}

@article{CS21,
  title={A note on norms of signed sums of vectors},
  author={Chasapis, Giorgos and Skarmogiannis, Nikos},
  journal={Advances in Geometry},
  volume={21},
  number={1},
  pages={5--14},
  year={2021},
  publisher={De Gruyter}
}

@article{Haj88,
  title={On a conjecture of Koml\'{o}s about signed sums of vectors inside the sphere},
  author={Hajela, D},
  journal={European Journal of Combinatorics},
  volume={9},
  number={1},
  pages={33--37},
  year={1988},
  publisher={Elsevier}
}

@article{Kun23,
  title={The discrepancy of unsatisfiable matrices and a lower bound for the Koml{\'o}s conjecture constant},
  author={Kunisky, Dmitriy},
  journal={SIAM Journal on Discrete Mathematics},
  volume={37},
  number={2},
  pages={586--603},
  year={2023},
  publisher={SIAM}
}

@article{Ban24,
  title={On a Generalization of Iterated and Randomized Rounding},
  author={Bansal, Nikhil},
  journal={Theory of Computing},
  volume={20},
  number={1},
  pages={1--23},
  year={2024},
  publisher={Theory of Computing Exchange}
}

@book{CST14,
  title={A panorama of discrepancy theory},
  author={Chen, William and Srivastav, Anand and Travaglini, Giancarlo},
  volume={2107},
  year={2014},
  publisher={Springer}
}

@inproceedings{Ban22,
  title={Discrepancy theory and related algorithms},
  author={Bansal, Nikhil},
  booktitle={Proc. Int. Cong. Math},
  volume={7},
  pages={5178--5210},
  year={2022}
}

@book{Cha00,
  title={The Discrepancy Method: Randomness and Complexity},
  author={Chazelle, Bernard},
  year={2000},
  publisher={Cambridge University Press}
}

@book{Mat09,
  author    = {Jiří Matoušek},
  title     = {Geometric Discrepancy: An Illustrated Guide},
  series    = {Algorithms and Combinatorics},
  volume    = {18},
  publisher = {Springer-Verlag},
  address   = {Berlin},
  year      = {1999},
}

@article{Spe85,
  title={Six standard deviations suffice},
  author={Spencer, Joel},
  journal={Transactions of the American mathematical society},
  volume={289},
  number={2},
  pages={679--706},
  year={1985}
}

@article{Glu89,
  title={Extremal properties of orthogonal parallelepipeds and their applications to the geometry of Banach spaces},
  author={Gluskin, Efim Davydovich},
  journal={Mathematics of the USSR-Sbornik},
  volume={64},
  number={1},
  pages={85},
  year={1989}
}

@article{Bec81,
  title={Roth’s estimate of the discrepancy of integer sequences is nearly sharp},
  author={Beck, J{\'o}zsef},
  journal={Combinatorica},
  volume={1},
  number={4},
  pages={319--325},
  year={1981},
  publisher={Springer}
}

@article{BF81,
  title={“{I}nteger-making” theorems},
  author={Beck, J{\'o}zsef and Fiala, Tibor},
  journal={Discrete Applied Mathematics},
  volume={3},
  number={1},
  pages={1--8},
  year={1981},
  publisher={Elsevier}
}

@article{Ban98,
  title={Balancing vectors and {G}aussian measures of n-dimensional convex bodies},
  author={Banaszczyk, Wojciech},
  journal={Random Structures \& Algorithms},
  volume={12},
  number={4},
  pages={351--360},
  year={1998},
  publisher={Wiley Online Library}
}

@InProceedings{BLV22,
  author =	{Bansal, Nikhil and Laddha, Aditi and Vempala, Santosh},
  title =	{{A Unified Approach to Discrepancy Minimization}},
  booktitle =	{Approximation, Randomization, and Combinatorial Optimization. Algorithms and Techniques (APPROX/RANDOM 2022)},
  pages =	{1:1--1:22},
  series =	{Leibniz International Proceedings in Informatics (LIPIcs)},
  ISBN =	{978-3-95977-249-5},
  ISSN =	{1868-8969},
  year =	{2022},
  volume =	{245},
  editor =	{Chakrabarti, Amit and Swamy, Chaitanya},
  publisher =	{Schloss Dagstuhl -- Leibniz-Zentrum f{\"u}r Informatik},
  address =	{Dagstuhl, Germany},
  URL =		{https://drops.dagstuhl.de/entities/document/10.4230/LIPIcs.APPROX/RANDOM.2022.1},
  URN =		{urn:nbn:de:0030-drops-171238},
  doi =		{10.4230/LIPIcs.APPROX/RANDOM.2022.1}
}

@inproceedings{BG17,
  title={Algorithmic discrepancy beyond partial coloring},
  author={Bansal, Nikhil and Garg, Shashwat},
  booktitle={ Symposium on Theory of Computing, {STOC}},
  pages={914--926},
  year={2017}
}

@article{BDG19,
  title={An algorithm for {K}oml{\'o}s conjecture matching Banaszczyk's bound},
  author={Bansal, Nikhil and Dadush, Daniel and Garg, Shashwat},
  journal={SIAM Journal on Computing},
  volume={48},
  number={2},
  pages={534--553},
  year={2019},
  publisher={SIAM}
}

\end{document}